\documentclass[12pt,english,a4paper]{smfart}

\usepackage[T1]{fontenc}
\usepackage{lmodern}
\usepackage{smfthm}
\usepackage[headings]{fullpage}
\usepackage{amssymb}

\newcommand{\im}{\operatorname{im}}

\newcommand{\Fil}{\operatorname{Fil}}
\newcommand{\bfont}{\mathbf{B}}
\newcommand{\bdr}{\mathbf{B}_{\mathrm{dR}}}

\newcommand{\Gal}{\operatorname{Gal}}

\newcommand{\Kbar}{\overline{K}}
\newcommand{\Cp}{\mathbf{C}}
\newcommand{\Zp}{\mathbf{Z}_p}
\newcommand{\ZZ}{\mathbf{Z}}

\newcommand{\OO}{\mathcal{O}}

\newcommand{\frD}{\mathfrak{d}}
\newcommand{\Ann}{\operatorname{Ann}}

\newcommand{\val}{\operatorname{val}}
\newcommand{\vp}{\val_p}
\newcommand{\Tr}{\operatorname{Tr}}
\newcommand{\Nm}{\operatorname{N}}

\newcommand{\smallo}{\operatorname{o}}

\renewcommand{\geq}{\geqslant}
\renewcommand{\leq}{\leqslant} 
\renewcommand{\phi}{\varphi} 
\renewcommand{\hat}{\widehat}

\author{Laurent Berger}
\address{UMPA de l'ENS de Lyon \\
UMR 5669 du CNRS}
\email{laurent.berger@ens-lyon.fr}
\urladdr{perso.ens-lyon.fr/laurent.berger/}

\title{K\"ahler differentials and $\Zp$-extensions}

\date{April 19, 2023}

\begin{document}

\begin{abstract}
Let $K$ be a $p$-adic field, and let $K_\infty/K$ be a Galois extension that is almost totally ramified, and whose Galois group is a $p$-adic Lie group of dimension $1$. We prove that $K_\infty$ is not dense in $(\mathbf{B}_{\mathrm{dR}}^+ / \operatorname{Fil}^2 \mathbf{B}_{\mathrm{dR}}^+ )^{\operatorname{Gal}(\overline{K}/K_\infty)}$. Moreover, the restriction of $\theta$ to the closure of $K_\infty$ is injective, and its image via $\theta$ is the set of vectors of $\widehat{K}_\infty$ that are $C^1$ with zero derivative for the action of $\operatorname{Gal}(K_\infty/K)$. The main ingredient for proving these results is the construction of an explicit lattice of $\mathcal{O}_{K_\infty}$ that is commensurable with $\mathcal{O}_{K_\infty}^{d=0}$, where $d : \mathcal{O}_{K_\infty} \to \Omega_{\mathcal{O}_{K_\infty} / \mathcal{O}_K}$ is the differential.
\end{abstract}



\maketitle

\tableofcontents

\setlength{\baselineskip}{18pt}

\section*{Introduction}

Let $K$ be a $p$-adic field, namely a finite extension of $W(k)[1/p]$ where $k$ is a perfect field of characteristic $p$. 
Let $\Cp$ be the $p$-adic completion of an algebraic closure $\Kbar$ of $K$. Let $K_\infty/K$ be a Galois extension that is almost totally ramified, and whose Galois group is a $p$-adic Lie group of dimension $1$. Let $\bdr(\hat{K}_\infty) = \bdr(\Cp)^{\Gal(\Kbar/K_\infty)}$ be Fontaine's field of periods attached to $K_\infty/K$, and for $n \geq 1$, let $\bfont_n(\hat{K}_\infty) = \bdr^+(\hat{K}_\infty)/\Fil^n \bdr^+(\hat{K}_\infty)$. 

This note is motivated by Ponsinet's paper \cite{P20}, in which he relates the study of universal norms for the extension $K_\infty/K$ to the question of whether $K_\infty$ is dense in $\bfont_n(\hat{K}_\infty)$ for $n \geq 1$. The density result holds for $n=1$ by the Ax-Sen-Tate theorem. Our main result is the following.

\begin{enonce*}{Theorem A}
The field $K_\infty$ is not dense in $\bfont_2(\hat{K}_\infty)$.
\end{enonce*}

By the constructions of Fontaine and Colmez (see \cite{F94} and \cite{C12}), $\bfont_2(\Cp)$ is the completion of $\Kbar$ for a topology defined using the  K\"ahler differentials $\Omega_{\OO_{\Kbar} / \OO_K}$. Some partial results towards theorem A have been proved by Iovita-Zaharescu in \cite{IZ99}, by studying these K\"ahler differentials. Let $\Omega_{\OO_{K_\infty} / \OO_K}$ be the K\"ahler differentials of $\OO_{K_\infty} / \OO_K$ and let $d : \OO_{K_\infty} \to \Omega_{\OO_{K_\infty} / \OO_K}$ be the differential. Our main technical result is the construction of a lattice of $\OO_{K_\infty}$ that is commensurable with $\OO_{K_\infty}^{d=0}$. Since the inertia subgroup of $\Gal(K_\infty/K)$ is a $p$-adic Lie group of dimension $1$, there exists a finite subextension $K_0/K$ of $K_\infty$ such that $K_\infty/K_0$ is a totally ramified $\Zp$-extension. Let $K_n$ be the $n$-th layer of this $\Zp$-extension. 

\begin{enonce*}{Theorem B}
The lattices $\sum_{n \geq 0} p^n \OO_{K_n}$ and $\OO_{K_\infty}^{d=0}$ are commensurable.
\end{enonce*}

In order to prove this, we use Tate's results on ramification in $\Zp$-extensions. As a corollary of theorem B, we can  say more about the completion of $K_\infty$ in $\bfont_2(\hat{K}_\infty)$. The field $\hat{K}_\infty$ is a Banach representation of the $p$-adic Lie group $\Gal(K_\infty/K)$. Let $\theta : \bfont_2(\Cp) \to \Cp$ be the usual map from $p$-adic Hodge theory.

\begin{enonce*}{Theorem C}
The completion of $K_\infty$ in $\bfont_2(\hat{K}_\infty)$ is isomorphic via $\theta$ to the set of vectors of $\hat{K}_\infty$ that are $C^1$ with zero derivative for the action of $\Gal(K_\infty/K)$.

This is a field, and it is also the set of $y \in \hat{K}_\infty$ that can be written as $y=\sum_{n \geq 0} p^n y_n$ with $y_n \in K_n$ and $y_n \to 0$.
\end{enonce*}

We also prove that $d(\OO_{K_\infty})$ contains no nontrivial $p$-divisible element (coro \ref{nopdiv}), and that $d :  \OO_{K_\infty} \to \Omega_{\OO_{K_\infty} / \OO_K}$ is not surjective (coro \ref{nosurj}). These two statements are equivalent to theorem A by the results of \cite{IZ99}; using our computations, we give a short independent proof.

\vspace{15pt}
\noindent\textbf{Acknowledgements.}
I thank L\'eo Poyeton for his remarks on an earlier version of this note.

\section{K\"ahler differentials}
\label{kdifsec}

Let $K$ be a $p$-adic field. If $L/K$ is a finite extension, let $\frD_{L/K} \subset \OO_L$ denote its different. 

\begin{prop}
\label{fontrec}
Let $K$ be a $p$-adic field, and let $L/K$ be an algebraic extension.
\begin{enumerate}
\item If $L/K$ is a finite extension, then $\Omega_{\OO_L / \OO_K} = \OO_L / \frD_{L/K}$ as $\OO_L$-modules.
\item If $M/L/K$ are finite extensions, the map $\Omega_{\OO_L / \OO_K} \to \Omega_{\OO_M / \OO_K}$ is injective.
\item If $L/K$ is an algebraic extension, and $\omega_1,\omega_2 \in \Omega_{\OO_L / \OO_K}$, then there exists $x \in \OO_L$ such that $\omega_2 = x \omega_1$ if and only if $\Ann(\omega_1) \subset \Ann(\omega_2)$.
\end{enumerate}
\end{prop}

\begin{proof}
See for instance \S 2 of \cite{F82}.
\end{proof}

Recall (see \S 2 of \cite{CG96}) that an algebraic extension $L/K$ is deeply ramified if the set $\{ \vp(\frD_{F/K}) \}_F$ is unbounded, as $F$ runs through the set of finite extensions of $K$ contained in $L$. Alternatively (remark 3.3 of \cite{S12}), $L/K$ is deeply ramified if and only if $\hat{L}$ is a perfectoid field. An extension $K_\infty/K$ as in the introduction is deeply ramified.

\begin{coro}
\label{deepramdiff}
If $L/K$ is deeply ramified, then $\Omega_{\OO_L / \OO_K} = L/\OO_L$ as $\OO_L$-modules
\end{coro}

\begin{prop}
\label{equivtheo}
If $L/K$ is deeply ramified, then $d : \OO_L  \to \Omega_{\OO_L / \OO_K}$ is surjective if and only if $d (\OO_L)$ is $p$-divisible.
\end{prop}

\begin{proof}
Since $L/K$ is deeply ramified, $\Omega_{\OO_L / \OO_K}$ is $p$-divisible. This proves one implication. Assume now that $d (\OO_L)$ is $p$-divisible, so that there exists a sequence $\{\alpha_i \}_{i \geq 1}$ of $\OO_L$ such that $d \alpha_1 \neq 0$ and $d \alpha_i = p \cdot d \alpha_{i+1}$ for all $i \geq 1$. If $\omega \in \Omega_{\OO_L / \OO_K}$, prop \ref{fontrec} implies that there exists $i \geq 1$ and $x_i \in \OO_L$ such that $\omega = x_i \cdot d \alpha_i$. Take $k \geq 0$ such that $d(p^k x_i)=0$. We then have $\omega = x_i \cdot d \alpha_i = p^k x_i \cdot d \alpha_{i+k} = d( p^k x_i \alpha_{i+k})$. 
Hence $d$ is surjective.
\end{proof}

\begin{prop}
\label{subunif}
Let $L/K$ be a deeply ramified extension, and let $K' \subset L$ be a finite extension of $K$.
\begin{enumerate}
\item $d : \OO_L  \to \Omega_{\OO_L / \OO_K}$ is surjective if and only if $d' : \OO_L  \to \Omega_{\OO_L / \OO_{K'}}$ is surjective.
\item $\OO_L^{d=0}$ and $\OO_L^{d'=0}$ are commensurable.
\end{enumerate}
\end{prop}

\begin{proof}
We have an exact sequence of $\OO_L$-modules, compatible with $d$ and $d'$
\[ \OO_L \otimes \Omega_{\OO_{K'} / \OO_K} \xrightarrow{f} \Omega_{\OO_L / \OO_K} \xrightarrow{g} \Omega_{\OO_L / \OO_{K'}} \to 0. \]
Let us prove (1). If $d : \OO_L  \to \Omega_{\OO_L / \OO_K}$ is surjective, then clearly $d' : \OO_L  \to \Omega_{\OO_L / \OO_{K'}}$ is surjective. Conversely, there exists $r \geq 0$ such that $p^r \cdot \Omega_{\OO_{K'} / \OO_K} = \{0\}$. If $\omega \in \Omega_{\OO_L / \OO_K}$, write it as $\omega = p^r \omega_r$. By hypothesis, there exists $\alpha_r \in \OO_L$ such that $\omega_r = d' \alpha_r$ in $\Omega_{\OO_L / \OO_{K'}}$. Hence $p^r(\omega_r - d\alpha_r) = 0$ in $\Omega_{\OO_L / \OO_K}$ so that $\omega = d(p^r \alpha_r)$. We now prove (2). The exact sequence above implies that $\OO_L^{d=0} \subset \OO_L^{d'=0}$. Conversely, if $x \in \OO_L^{d'=0}$, then $dx \in \ker g = \im f$, so that $p^r \cdot d x = 0$. Hence $p^r \cdot \OO_L^{d'=0} \subset \OO_L^{d=0}$.
\end{proof}

\begin{coro}
\label{reduczp}
In order to prove theorem B, we can replace $K$ by any finite subextension $K'$ of $K$. In particular, we can assume that $K_\infty/K$ is a totally ramified $\Zp$-extension. 
\end{coro}

\section{Ramification in $\Zp$-extensions}
\label{zpsec}

Let $K_\infty/K$ be a totally ramified $\Zp$-extension. We recall some of the results of \S 3.1 of \cite{T67} concerning the ramification of $K_\infty/K$ and the action of $\Gal(K_\infty/K)$ on $K_\infty$. Let $K_n$ be the $n$-th layer of $K_\infty/K$, so that $[K_n:K]=p^n$.

\begin{prop}
\label{tatediff}
There are constants $a,b$ such that $|{\vp(\frD_{K_n/K}) - n - b}| \leq p^{-n} a$ for $n \geq 0$. 
\end{prop}

\begin{proof}
See \S 3.1 of \cite{T67}.
\end{proof}

The notation $\sum_{n \geq 0} p^n \OO_{K_n}$ denotes the set of elements of $K_\infty$ that are finite sums of elements of $p^n \OO_{K_n}$.

\begin{coro}
\label{trivinc}
There exists $n_0 \geq 0$ such that $\sum_{n \geq 0} p^{n+n_0} \OO_{K_n} \subset \OO_{K_\infty}^{d=0}$.
\end{coro}

\begin{prop}
\label{fonemb}
There exists $c(K_\infty/K) > 0$ such that for all $n,k \geq 0$ and $x \in \OO_{K_{n+k}}$, we have $\vp(\Nm_{K_{n+k}/K_n}(x)/x^{[K_{n+k}:K_n]}-1) \geq c(K_\infty/K)$.
\end{prop}

\begin{proof}
The result follows from the fact (see 1.2.2 of \cite{W83}) that the extension $K_\infty/K$ is strictly APF. One can then apply 1.2.1, 4.2.2 and 1.2.3 of \cite{W83}.
\end{proof}

If $n \geq 0$ and $x \in K_\infty$, then $R_n(x) = p^{-k} \cdot \Tr_{K_{n+k}/K_n} (x)$ is independent of $k \gg 0$ such that $x \in K_{n+k}$, and is the normalized trace of $x$.

\begin{prop}
\label{rnbdd}
There exists $c_2 \in \ZZ_{\geq 0}$ such that $\vp(R_n(x)) \geq \vp(x)-c_2$ for all $n \geq 0$ and $x \in K_\infty$.
\end{prop}

\begin{proof}
See \S 3.1 of \cite{T67} (including the remark at the bottom of page 172).
\end{proof}

In particular, $R_n(\OO_{K_\infty}) \subset p^{-c_2} \OO_{K_n}$ for all $n \geq 0$. Let $K_0^\perp = K_0$ and for $n \geq 1$, let $K_n^\perp$ be the kernel of $R_{n-1} : K_n \to K_{n-1}$, let $R_n^\perp = R_n - R_{n-1}$, and $R_0^\perp = R_0$. If $x \in K_\infty$ and $i \geq 0$, then $R_n^\perp(x) = 0$ for $n \gg 0$, and $x = (\sum_{n \geq i+1} R_n^\perp(x)) + R_i(x)$. Prop \ref{rnbdd} implies that $R_n^\perp(\OO_{K_\infty}) \subset p^{-c_2} \OO_{K_n}$ for all $n \geq 0$. Let $\OO_{K_n}^\perp = \OO_{K_n} \cap K_n^\perp$.

\begin{coro}
\label{decper}
For all $i \geq 0$, we have $\OO_{K_\infty} \subset (\oplus_{m \geq i+1} p^{-c_2} \OO_{K_m}^\perp) \oplus p^{-c_2}\OO_{K_i}$.
\end{coro}

\begin{proof}
If $x \in \OO_{K_\infty}$, write $x=\sum_{m \geq i+1} R_m^\perp(x) +R_i(x)$.
\end{proof}

For $n \geq 0$, let $g_n$ denote a topological generator of $\Gal(K_\infty/K_n)$.

\begin{lemm}
\label{gaminv}
There exists a constant $c_3$ such that for all $n \geq 1$ and $x \in K_n^{\perp}$, we have $\vp(x) \geq \vp((1-g_n)(x))-c_3$.
\end{lemm}

\begin{proof}
See \S 3.1 of \cite{T67} (including the remark at the bottom of page 172).
\end{proof}

\section{The lattice $\OO_{K_\infty}^{d=0}$}
\label{lattsec}

We now prove theorem B. Thanks to coro \ref{reduczp}, we assume that $K_\infty/K$ is a totally ramified $\Zp$-extension. Let $\{\rho_n\}_{n \geq 0}$ be a norm compatible sequence of uniformizers of the $K_n$. Let $m_c \geq 0$ be the smallest integer such that $p^{m_c} \cdot c(K_\infty/K) \geq 1/(p-1)$.

\begin{prop}
\label{rhoval}
We have $\vp(\rho_{n+1}^{pk} - \rho_n^k) \geq  \vp(k)-m_c$.
\end{prop}

\begin{proof}
Note that if $x,y \in \Cp$ with $\vp(x-y) \geq v$, then $\vp(x^p-y^p) \geq \min (v+1,pv)$. Let $c=c(K_\infty/K)$ and $m=m_c$. We have $\vp(\rho_{n+1}^p - \rho_n) \geq c$ by prop \ref{fonemb}, so that $\vp(\rho_{n+1}^{p^{j+1}} - \rho_n^{p^j}) \geq p^j c$ for all $j$ such that $p^{j-1} c \leq 1/(p-1)$. 

In particular, $\vp(\rho_{n+1}^{p^{m+1}} - \rho_n^{p^m}) \geq p^m c \geq 1/(p-1)$, so that $\vp(\rho_{n+1}^{p^{m+j+1}} - \rho_n^{p^{m+j}}) \geq j+1/(p-1)$ if $j \geq 0$. This implies the result.
\end{proof}

\begin{theo}
\label{fouvar}
There exists $n_1 \in \ZZ_{\geq 0}$ such that $\OO_{K_\infty}^{d=0} \subset \sum_{m \geq n_1} p^{m-n_1} \OO_{K_m}$.
\end{theo}

\begin{proof}
Let $n_1=\lceil a-b+m_c+2 \rceil$. Take $x \in \OO_{K_n}^{d=0}$ and write $x = \sum_{i=0}^{p^n-1} x_i \rho_n^i$ with $x_i \in \OO_K$, so that $dx = \sum_{i=0}^{p^n-1} i x_i \rho_n^{i-1} \cdot d \rho_n$. If $dx=0$, then $\sum_{i=0}^{p^n-1} i x_i \rho_n^{i-1} \in \frD_{K_n/K}$ so that by prop \ref{tatediff} (and since $\vp(\rho_n^{p^n}) \leq 1$), for all $i$ we have \[ \vp(x_i) \geq n-a+b-\vp(i)-1. \] 
For $k \geq 1$, let \[ y_k = \sum_{p \nmid j} x_{p^{k-1} j} \rho_{n-(k-1)}^j + \sum_{\ell} x_{p^k \ell} (\rho_{n-(k-1)}^{p \ell} - \rho_{n-k}^\ell). \]
Note that $y_k \in \OO_{K_{n-k+1}}$. Let us bound $\vp(y_k)$. We have $\vp (x_{p^{k-1} j} \rho_{n-(k-1)}^j) \geq n-a+b-k$. We also have $\vp(x_{p^k \ell}) \geq n-a+b-k-\vp(\ell)-1$, and $\vp(\rho_{n-(k-1)}^{p \ell} - \rho_{n-k}^\ell) \geq \vp(\ell)-m_c$ by prop \ref{rhoval}. Hence $\vp(y_k) \geq n-a+b-k-1-m_c$ and therefore $y_k \in p^{n-k+1-n_1} \OO_{K_{n-k+1}}$. Finally, we have $x=y_1+ \cdots + y_{n-n_1} + \sum_\ell x_{p^{n-n_1} \ell} \rho_{n_1}^\ell$, and $\sum_\ell x_{p^{n-n_1} \ell} \rho_{n_1}^\ell \in \OO_{K_{n_1}}$, which implies the result.
\end{proof}

\begin{rema}
\label{fouvarem}
Compare with lemma 4.3.2 of \cite{F05}.
\end{rema}

\begin{coro}
\label{perpdec}
We have $\OO_{K_\infty}^{d=0} \subset (\oplus_{m \geq n_1+1} p^{m-n_1-c_2} \OO_{K_m}^\perp) \oplus p^{-c_2}\OO_{K_{n_1}}$.
\end{coro}

\begin{proof}
By theorem \ref{fouvar}, it is enough to prove that $p^n \OO_{K_n} \subset (\oplus_{m \geq n_1+1} p^{m-c_2} \OO_{K_m}^\perp) \oplus p^{n_1-c_2}\OO_{K_{n_1}}$ for all $n \geq n_1$. If $x \in p^n \OO_{K_n}$, write $x = R_n^\perp(x) + R_{n-1}^\perp(x)  + \cdots + R_{n_1+1}^\perp(x) + R_{n_1}(x)$. We have $R_{n-k}^\perp(x) \in p^{n-c_2} \OO_{K_{n-k}}^\perp \subset p^{(n-k)-c_2} \OO_{K_{n-k}}^\perp$ and likewise $R_{n_1}(x) \in p^{n-c_2} \OO_{K_{n_1}} \subset  p^{n_1-c_2}\OO_{K_{n_1}}$.
\end{proof}

\begin{coro}
\label{nopdiv}
There are no nontrivial $p$-divisible elements in $d (\OO_{K_\infty})$.
\end{coro}

\begin{proof}
By props \ref{equivtheo} and \ref{subunif}, we can assume that $K_\infty/K$ is a totally ramified $\Zp$-extension.
Let $\{ \alpha_i \}_{i \geq 1}$ be a sequence of $\OO_{K_\infty}$ such that $d \alpha_i = p \cdot d\alpha_{i+1}$ for all $i \geq 1$. Write $\alpha_i= \sum \alpha_{i,m}$ with $\alpha_{i,m} \in p^{-c_2} \OO_{K_m}^\perp$ for $m \geq n_1+1$ and $\alpha_{i, n_1} \in p^{-c_2}\OO_{K_{n_1}}$. Since $p^k \alpha_{k+i} - \alpha_i \in \OO_{K_\infty}^{d=0}$, coro \ref{perpdec} implies that $p^k \alpha_{k+i,m} - \alpha_{i,m} \in p^{m-n_1-c_2} \OO_{K_m}$ for all $m \geq n_1$. Taking $k \gg 0$ now implies that $\alpha_{i,m} \in p^{m-n_1-c_2} \OO_{K_m}$ for all $m \geq n_1$. Coro \ref{trivinc} gives $p^{n_0+n_1+c_2} \alpha_i \in \OO_{K_\infty}^{d=0}$. Taking $i=n_0+n_1+c_2+1$ gives $d \alpha_1 = 0$.
\end{proof}

\begin{coro}
\label{nosurj}
The differential $d :  \OO_{K_\infty} \to \Omega_{\OO_{K_\infty} / \OO_K}$ is not surjective.
\end{coro}

\begin{proof}
This follows from coro \ref{nopdiv} and prop \ref{equivtheo}.
\end{proof}

\section{The completion of $K_\infty$ in $\bfont_2$}
\label{compsec}

We now prove theorems A and C. Since we are concerned with the completion of $K_\infty$, we can once again replace $K$ with a finite subextension of $K_\infty$ and assume that $K_\infty/K$ is a totally ramified $\Zp$-extension. Let $\hat{K}_\infty^2$ denote the completion of $K_\infty$ in $\bfont_2$, so that $R = \theta(\hat{K}_\infty^2)$ is a subring of $\hat{K}_\infty$. Let $\Gamma = \Gal(K_\infty/K)$, and let $c : \Gamma \to \Zp$ be an isomorphism of $p$-adic Lie groups. Let $w_2$ be the valuation on $K_\infty$ defined by $w_2(x) = \min \{ n \in \ZZ$ such that $p^n x \in \OO_{K_\infty}^{d=0}\}$. The restriction of the natural valuation of $\bfont_2$ to $K_\infty$ is $w_2$ (see \S 1.4 and \S 1.5 of \cite{F94}, or theorem 3.1 of \cite{C12}).

\begin{lemm}
\label{thet}
If $\{x_k\}_{k \geq 1}$ is a sequence of $K_\infty$ that converges to $x \in \bfont_2$ for $w_2$, then $\{x_k\}_{k \geq 1}$ is Cauchy for $\vp$, and $\theta(x) = \lim_{k \to +\infty} x_k$ for the $p$-adic topology.
\end{lemm}

Let $M = \oplus_{n \geq 0} p^n \OO_{K_n}^\perp$. Coro \ref{trivinc} and theo \ref{fouvar} imply that $M$ and $\OO_{K_\infty}^{d=0}$ are commensurable. Hence $\hat{K}_\infty^2$ is the $M$-adic completion of $K_\infty$. Let $w'_2$ be the $M$-adic valuation on $K_\infty$, so that $w'_2$ and $w_2$ are equivalent.

\begin{lemm}
\label{rnw}
If $x \in K_\infty$, then $\vp(R_n^\perp(x)) \geq w'_2(x) +n$.
\end{lemm}

\begin{proof}
Write $x = \sum_{n \geq 0} R_n^\perp(x)$. If $x \in p^w M$, then  $R_n^\perp(x) \in p^{n+w} \OO_{K_n}$.
\end{proof}

\begin{prop}
\label{rnk2}
Every element $x \in \hat{K}_\infty^2$ can be written in one and only one way as $\sum_{n \geq 0} x_n^\perp$ where $x_n^\perp \in K_n^\perp$ and $p^{-n} x_n^\perp \to 0$ for $\vp$.
\end{prop}

\begin{proof}
Note that such a series converges for $w_2$. The map $R_n^\perp : K_\infty \to K_n^\perp$ sends $p^w M \subset K_\infty$ to $p^{w+n} \OO_{K_n}^\perp$, so that it is uniformly continuous for the $w_2$-adic topology, and therefore extends to a continuous map $R_n^\perp : \hat{K}_\infty^2 \to K_n^\perp$.

Let $x \in \hat{K}_\infty^2$ be the $w_2$-adic limit of $\{x_k\}_{k \geq 1}$ with $x_k \in K_\infty$. For a given $k$, the sequence $\{p^{-n} R_n^\perp(x_k)\}_{n \geq 0} \in \prod_{n \geq 0} K_n^\perp$ has finite support. As $k \to +\infty$, these sequences converge uniformly in $\prod_{n \geq 0} K_n^\perp$ to $\{p^{-n} R_n^\perp(x)\}_{n \geq 0}$, so that $p^{-n} R_n^\perp(x) \to 0$ as $n \to +\infty$. Hence $\sum_{n \geq 0} R_n^\perp(x)$ converges for $w_2$. Since $x_k = \sum_{n \geq 0} R_n^\perp(x_k)$ for all $k$, we have $x = \sum_{n \geq 0} R_n^\perp(x)$. Finally, if $x = \sum_{n \geq 0} x_n^\perp$ with $x_n^\perp \in K_n^\perp$ and $p^{-n} x_n^\perp \to 0$ for $\vp$, then $x_n^\perp = R_n^\perp(x)$ which proves unicity.
\end{proof}

\begin{coro}
\label{thetinj}
The map $\theta : \hat{K}_\infty^2 \to \hat{K}_\infty$ is injective.
\end{coro}

\begin{proof}
If $x_n^\perp \in K_n^\perp$ and $x_n^\perp \to 0$ and $\sum_{n \geq 0} x_n^\perp = 0$ in $\hat{K}_\infty$, then $x_n^\perp=0$ for all $n$.
\end{proof}

\begin{coro}
\label{imcomp}
The ring $R$ is the set of $y \in \hat{K}_\infty$ that can be written as $y=\sum_{n \geq 0} p^n y_n$ with $y_n \in K_n$ and $y_n \to 0$.
\end{coro}

\begin{prop}
\label{diffvec}
The ring $R$ is a field, and $R = \{ x \in \hat{K}_\infty$ such that $g(x)-x = \smallo(c(g))$ as $g \to 1$ in $\Gamma \}$.
\end{prop}

\begin{proof}
The fact that $R$ is a field results from the second statement, since $g(1/x)-1/x = (x-g(x))/(x g(x))$. 
Take $y=\sum_{n \geq 0} p^n y_n$ with $y_n \in K_n$ and $y_n \to 0$. If $m \geq 1$, then for all $k \gg 0$, we have $y_n \in p^{m+n} \OO_{K_n}$. We can write $y= x_k + \sum_{n \geq k} p^n y_n$ and then $(g-1)(y) \in p^{k+m} \OO_{K_\infty}$ if $g \in \Gal(K_\infty/K_k)$. This proves one implication.

Conversely, take $x \in \hat{K}_\infty$ such that $g(x)-x = \smallo(c(g))$. Write $x= \sum_{k \geq 0} x_k$ with $x_0 = R_0(x) \in K_0$ and $x_k = R_k^\perp(x) \in K_k^\perp$ for all $k \geq 1$. For $n \geq 0$, let $g_n$ denote a topological generator of $\Gal(K_\infty/K_n)$. Take $m \geq 0$ and $n \gg 0$ such that we have $\vp((g_n-1)(x)) \in p^{m+n} \OO_{K_\infty}$. We have $(1-g_n)(x) = \sum_{k \geq n+1} (1-g_n) x_k$, so that by lemma \ref{gaminv} and prop \ref{rnbdd}: $\vp(x_{n+1}) \geq \vp((1-g_n)(x_{n+1})) - c_3 \geq \vp((1-g_n)(x))-c_2-c_3 \geq n+m-c_2-c_3$. 
This implies the result.
\end{proof}

\begin{rema}
\label{diffvecrem}
Prop \ref{diffvec} says that $R$ is the set of vectors of $\hat{K}_\infty$ that are $C^1$ with zero derivative (flat to order $1$) for the action of $\Gamma$.
\end{rema}

Theorem A follows from coro \ref{thetinj} since $\theta : \bfont_2(\hat{K}_\infty) \to \hat{K}_\infty$ is not injective. Finally, coro \ref{thetinj}, coro \ref{imcomp}, and prop \ref{diffvec} imply theorem C.

\providecommand{\bysame}{\leavevmode ---\ }
\providecommand{\og}{``}
\providecommand{\fg}{''}
\providecommand{\smfandname}{\&}
\providecommand{\smfedsname}{\'eds.}
\providecommand{\smfedname}{\'ed.}
\providecommand{\smfmastersthesisname}{M\'emoire}
\providecommand{\smfphdthesisname}{Th\`ese}


\begin{thebibliography}{Win83}

\bibitem[CG96]{CG96}
{\scshape J.~Coates {\normalfont \smfandname} R.~Greenberg} -- {\og Kummer
  theory for abelian varieties over local fields\fg}, \emph{Invent. Math.}
  \textbf{124} (1996), no.~1-3, p.~129--174.

\bibitem[Col12]{C12}
{\scshape P.~Colmez} -- {\og Une construction de
  {$\mathbf{B}_{\mathrm{dR}}^+$}\fg}, \emph{Rend. Semin. Mat. Univ. Padova}
  \textbf{128} (2012), p.~109--130 (2013).

\bibitem[Fon94]{F94}
{\scshape J.-M. Fontaine} -- {\og Le corps des p\'{e}riodes {$p$}-adiques\fg},
  \emph{Ast\'{e}risque} (1994), no.~223, p.~59--111, With an appendix by Pierre
  Colmez, P\'{e}riodes $p$-adiques (Bures-sur-Yvette, 1988).

\bibitem[Fon82]{F82}
\bysame , {\og Formes diff\'{e}rentielles et modules de {T}ate des
  vari\'{e}t\'{e}s ab\'{e}liennes sur les corps locaux\fg}, \emph{Invent.
  Math.} \textbf{65} (1981/82), no.~3, p.~379--409.

\bibitem[Fou05]{F05}
{\scshape L.~Fourquaux} -- {\og Logarithme de {P}errin-{R}iou pour des
  extensions associ\'ees \`a un groupe de {L}ubin-{T}ate\fg},
  \smfphdthesisname, Universit\'e {P}aris 6, 2005.

\bibitem[IZ99]{IZ99}
{\scshape A.~Iovita {\normalfont \smfandname} A.~Zaharescu} -- {\og Galois
  theory of {$B_{\rm dR}^+$}\fg}, \emph{Compositio Math.} \textbf{117} (1999),
  no.~1, p.~1--31.

\bibitem[Pon20]{P20}
{\scshape G.~Ponsinet} -- {\og Universal norms and the {F}argues-{F}ontaine
  curve\fg}, preprint, 2020.

\bibitem[Sch12]{S12}
{\scshape P.~Scholze} -- {\og Perfectoid spaces\fg}, \emph{Publ. Math. Inst.
  Hautes \'{E}tudes Sci.} \textbf{116} (2012), p.~245--313.

\bibitem[Tat67]{T67}
{\scshape J.~T. Tate} -- {\og {$p$}-divisible groups\fg}, in \emph{Proc.
  {C}onf. {L}ocal {F}ields ({D}riebergen, 1966)}, Springer, Berlin, 1967,
  p.~158--183.

\bibitem[Win83]{W83}
{\scshape J.-P. Wintenberger} -- {\og Le corps des normes de certaines
  extensions infinies de corps locaux; applications\fg}, \emph{Ann. Sci.
  \'{E}cole Norm. Sup. (4)} \textbf{16} (1983), no.~1, p.~59--89.

\end{thebibliography}
\end{document}